\pdfoutput=1

\documentclass[11pt]{amsart}
\usepackage{amsmath,amssymb}
\usepackage{graphicx}
\usepackage{amsfonts,amscd,latexsym}

\newcommand{\labell}[1] {\label{#1}}

\newcommand{\1}{{{\mathchoice {\rm 1\mskip-4mu l} {\rm 1\mskip-4mu l}
{\rm 1\mskip-4.5mu l} {\rm 1\mskip-5mu l}}}}

\newlength{\facewd} \newlength{\faceht}%
\newcommand{\eye}[1]{%
\settowidth{\facewd}{#1}\settoheight{\faceht}{#1}%
\raisebox{1.1\faceht}{%
\makebox[0pt][l]{%
$\hspace{.3\facewd}\scriptscriptstyle\circ$}}%
#1}%

\newcommand{\ooDe}{{\overset{\scriptscriptstyle\circ}\De}}
\newcommand{\ooE}{{\overset{\scriptscriptstyle\circ}E}}
\newcommand{\oB}{\eye{B}}
\newcommand{\oE}{{\eye{E}}}
\newcommand{\bx} {{\bf x}}
\newcommand{\ve}{{\vareps}}
\renewcommand{\Hat}{\widehat}

\newcommand{\SL}{{\rm SL}}

\newcommand{\less}{{\smallsetminus}}

\newcommand{\un}{\underline}

\newcommand{\al}{{\alpha}}

\newcommand{\be}{{\beta}}

\newcommand{\Om}{{\Omega}}
\newcommand{\om}{{\omega}}
\newcommand{\eps}{{\varepsilon}}
\newcommand{\vareps}{{\epsilon}}
\newcommand{\de}{{\delta}}
\newcommand{\De}{{\Delta}}
\newcommand{\ga}{{\gamma}}

\newcommand{\io}{{\iota}}
\newcommand{\ka}{{\kappa}}
\newcommand{\la}{{\lambda}}

\newcommand{\Cc}{{\mathcal C}}

\newcommand{\Jj}{{\mathcal J}}
\newcommand{\Ff}{{\mathcal F}}

\newcommand{\Ee}{{\mathcal E}}

\newcommand{\ov}{\overline}

\newcommand{\ts}{\textstyle}

\newcommand{\PP}{{\mathbb P}}
\newcommand{\N}{{\mathbb N}}
\newcommand{\Q}{{\mathbb Q}}
\newcommand{\R}{{\mathbb R}}
\newcommand{\C}{{\mathbb C}}

\newcommand{\Z}{{\mathbb Z}}
\newcommand{\CP}{{\mathbb CP}}

\newcommand{\Nn}{{\mathcal N}}

\newcommand{\SSS}{{\smallskip}}
\newcommand{\QED}{{\hfill $\Box$\MS}}
\newcommand{\se} {{\stackrel{s}\hookrightarrow}}

\newtheorem{theorem}{Theorem}[section]
\newtheorem{thm}[theorem]{Theorem}

\newtheorem{cor}[theorem]{Corollary}
\newtheorem{lemma}[theorem]{Lemma}

\newtheorem{prop}[theorem]{Proposition}

\newtheorem{defn}[theorem]{Definition}
\newtheorem{example}[theorem]{Example}

\newtheorem{rmk}[theorem]{Remark}

\numberwithin{figure}{section}
\numberwithin{equation}{section}
\numberwithin{table}{section}

\newcommand{\MS}{{\medskip}}

\newcommand{\NI}{{\noindent}}
\begin{document}
 \title{Symplectic embeddings of $4$-dimensional ellipsoids}
 \author{Dusa McDuff}\thanks{partially supported by NSF grant DMS 0604769.}
\address{Department of Mathematics,
Barnard College, Columbia University, New York, NY 10027-6598, USA.}
\email{dusa@math.columbia.edu}
\keywords{symplectic embedding, symplectic packing problem,
weighted projective space, symplectic ellipsoid, symplectic capacity}
\subjclass[2000]{53D05}
\date{24 March 2008, revised 20 November 2008}
\begin{abstract} We show how to reduce the problem
of symplectically embedding one $4$-dimensional rational ellipsoid into another to a problem of embedding disjoint unions of balls into appropriate blow ups of $\C P^2$.
For example, the problem of embedding the ellipsoid  $E(1,k)$ into a ball $B$ is equivalent to that of  embedding $k$ disjoint equal balls into $\C P^2$, and so can be solved by the work of
Gromov, McDuff--Polterovich and Biran.
(Here $k$  is the ratio of the area of the major axis to that of the minor axis.) 
As a consequence we show that the ball  may be fully 
filled by the ellipsoid $E(1,k)$ for $k=1,4$ and all $k\ge 9$, thus answering a question raised by Hofer.  \end{abstract}
\maketitle

\section{Introduction.}

A $4$-dimensional symplectic ellipsoid is a region in standard Euclidean space $(\R^{4}, \om_0)$ described by an inequality of the form  $Q(z)\le 1$,
where  $Q$ is a positive definite quadratic form and $z\in \R^{4}$.
Since $Q$ may be diagonalized by  a linear change of coordinates,
 every ellipsoid may be written (uniquely) as $E(m,n)$ where
$$
E(m,n): = \Bigl\{\frac {x_1^2+y_1^2}{m} + \frac {x_2^2+y_2^2}{n}\le 1\Bigr\}\;\subset\;\R^4, \quad m\le n.
$$
We denote an open ellipsoid by $\oE$ and the ball $E(m,m)$ by $B(m)$. Further, $\la E$ denotes $E$
with the rescaled form $\la\om_0$.  
Thus $\la E(m,n) : = E(\la m,\la n)$.
Throughout  this paper the word \lq\lq embedding" will be used to denote a symplectic embedding.  If $E$ embeds in $(M,\om)$ we shall write
$E\se (M,\om)$. 

This paper is concerned with the question of when it is possible to embed one $4$-dimensional ellipsoid into another. 
There are two known obstructions to 
 embedding $E(m,n)$ into  $E(m',n')$ when $m,n, m',n'$ are integers; namely, if  such an embedding exists, we must have
\MS

\NI (i) 
 (Volume): $mn\le m'n'$ (since $E(m,n)$ has volume $\pi^2 mn/2$); and\SSS

\NI(ii)  (Ekeland--Hofer capacities): $N(m,n)\le N(m',n')$.
\MS

\NI  Here   $N(m,n)$ denotes the sequence  obtained by arranging
the numbers $km, k\ge 1,$ and $\ell n,\ell \ge 1,$ in nondecreasing order (with repetitions) and $N(m,n)\le N(m',n')$ means that every number in $N(m,n)$ is no larger than the corresponding number in $N(m',n')$; see \cite{CHLS}. In particular, we must have $m\le m'$ as follows from Gromov's nonsqueezing theorem.
For example, 
$$
N(1,4) = (1,2,3,4,4,5,6,7,8,8,9,\dots),\quad 
N(2,2) = (2,2,4,4,6,6,8,8,\dots),
$$
so that $N(1,4)\le N(2,2)$. Since the  volume inequality is also satisfied in this case the question arose as to whether  $E(\la,4\la)$ embeds in the ball  
$B(2)$ for all  $\la<1$.  (By  Corollary~\ref{cor:conn} this is equivalent to  asking if the interior  
$\oE(1,4)$  embeds into $B(2)$.)

When considering this embedding problem it is convenient to 
consider the maximal packing radius
$$
\la_{sup}: =\sup\; \bigl\{ \la \,|\, \la E( m, n)\se E(m',n')\bigr\}
$$
and the packing constant
$$
v: = \la_{sup}^2\, \frac  {mn}{m'n'} \;\le \;1,
$$
which is the ratio of the volume of the domain to that of the target.
Note that both these numbers are scale invariant, i.e. do not change if 
all numbers $m,n,m',n'$ are multiplied by the same constant $\mu$.
We say that $E(m,n)$ {\it fully fills} $E(m',n')$ if $v=1$.
Further we denote by $\C P^2(\mu)$ the complex projective plane with its standard Fubini--Study form, normalized so that the area of a line is $\mu$.
It is obtained from the ball $B(\mu)$ by collapsing 
its boundary sphere to a line.

\begin{thm}\labell{thm:1}  $E(1,k)$ embeds in  the open ball $\oB(\mu)$ if and only if 
$k$ disjoint  balls $B(1)$ embed in $\oB(\mu)$.
\end{thm}

The \lq\lq only if" part of this statement was first observed by Traynor \cite{T}.  It  is very easy to prove using toric models, which make it immediately clear that $E(1,k)$ 
 contains $k$ disjoint open balls $\oB(1)$: see Fig. \ref{fig:5}
and Lemma~\ref{le:open}. However the \lq\lq if" part requires more work. The main idea  
is to cut  the ellipsoid from the ball (i.e. perform an orbifold blow up as in Godinho \cite{God})  and then  to resolve the resulting orbifold 
singularities by further standard blow ups.    The necessary symplectic 
surgery techniques were developed by Symington \cite{Sym} in her treatment  of rational blowdowns.

\begin{cor}\label{cor:1} $E(1,k)$ fully fills $B(\sqrt k)$ if and only if $k=1,4$ or $k\ge 9$.  Moreover for $k\le 8$ the packing constant 
$v(k)$ is:
$$
\begin{array}{|c|c|c|c|c|c|c|c|c|}\hline
k&1&2&3&4&5&6&7&8\\\hline
v(k)&1&\frac 12&\frac 34 &1&\frac 45 &\frac{24}{25} &\frac {63}{64}&\frac {288}{289}
\\\hline
\end{array}
$$
\end{cor}
\NI {\it Proof of Corollary.} In his foundational paper \cite{G},
Gromov showed that $v(2)\le 1/2$ and $v(5)\le 4/5$.
When $k\le9$ or $k=d^2$, the rest of the above statement
follows from Theorem \ref{thm:1} by 
McDuff--Polterovich \cite{MP}.
The case $k>9$ is due to
Biran \cite{B}.
\QED
 
For explicit realizations of the ball 
packings at the integers $k<9$, see Karshon \cite{K}, Traynor \cite{T}, Schlenk \cite{Sch0},
and Wieck \cite{W}.

 \begin{rmk}\rm  After this paper was written, I found out 
that Opshtein's paper \cite{Op} on maximal symplectic 
 packings of $\C P^2$ contains a proof that $E(1,k)$ fully 
 fills $\C P^2$ when $k=d^2$. Though not stated 
 explicitly in his paper,
 this result follows immediately from his Lemma 2.1. His argument has 
the virtue of providing a geometric recipe for constructing these packings.  
 \end{rmk}

In \cite[Problem~15]{CHLS}  Cieliebak, Hofer, Latschev, and Schlenk ask whether the two invariants listed above are the only obstructions to embedding one open 
ellipsoid into another. More formally, they asked if
 the Ekeland--Hofer capacities $N(m,n)$ together with the volume capacity $V$ generate the (generalized)  symplectic capacities 
 on the space ${\it Ell}^4$ of open $4$-dimensional ellipsoids.
 To understand what this means, consider the
 capacity given by embeddings into an open ball:
 $$
 c^B(M,\om) : = \inf\bigl\{\mu>0 \,|\, (M,\om)\se \oB(\mu)\bigr\}.
 $$
If this capacity were some combination of $N(m,n)$ and $V$ for $(M,\om) = \oE(m,n)$  then
every time the inequalities given by these capacities are satisfied there would be an embedding of $\oE(m,n)$ into $\oB(\mu)$.
  But this is not so.  For example,
 $c^B(\oE(1,5)) = \frac 52$ because by Corollary~\ref{cor:1} the volume of the target ball must be
 at least $5/4$ times the volume of $\oE(1,5)$.  On the other hand,
 the volume and Ekeland--Hofer capacities give no obstruction to the existence of an embedding of $\oE(1,5)$  into $\la \oB(\sqrt 5)$ for all $\la>1$. We conclude:

 \begin{cor}\label{cor:2}
The Ekeland--Hofer capacities $N(m,n)$ together with the volume capacity do not generate the (generalized)  symplectic capacities on the space ${\it Ell}^4$ of open $4$-dimensional ellipsoids. 
\end{cor}

Note that this follows from
 the easy (only if)  part of Theorem~\ref{thm:1}.

Our second set of results concern the problem of embedding one ellipsoid into another.  It is convenient to introduce the following terminology. Given a positive integer $k$ and $k$ positive numbers $w_1,\dots,w_k$ the {\it (symplectic) packing problem for $k$ balls with weights} $\un w: = (w_1,\dots,w_k)$ is the question of whether the $k$ disjoint (closed) balls $B(w_1),\dots, B(w_k)$ embed into the open ball $\oB(1)$.

\begin{thm}\labell{thm:ell}   For any positive integers $m,n,m',n'$ and any $\la>0$, there is an integer $k$ and weights ${\un w}_\la$ such that  the question of whether $\la E( m, n)$ embeds into the open ellipsoid $\oE(m',n')$ is equivalent to the symplectic packing problem for $k$ balls with weights ${\un w}_\la$.
\end{thm}

The following result (which was proved  for balls in \cite{Mcuniq})  is an easy consequence.

\begin{cor}\labell{cor:conn} Let $a,b,a',b'$ be any real positive numbers.
Then:\SSS

\NI
 {\rm (i)} the space of symplectic embeddings of $E(a,b)\se \oE(a',b')$ is path connected whenever it is nonempty;\SSS

\NI {\rm (ii)}  if $\la E(a,b)$ embeds in $\oE(a',b')$ for all $\la<1$, then $\oE(a,b)$ also embeds in $\oE(a',b')$.
\end{cor}

We also work out two specific examples that answer a
 question raised by Tolman \cite[\S1]{Tol}. 
 
\begin{prop} \labell{prop:tol}{\rm (i)}
The question of whether
$\la E(1, 4)$ embeds in $\oE(2,3)$ is equivalent to the packing problem with $k=7$ and $\un w=\frac 13 (1,1,1,\la,\la,\la,\la)$. Hence the embedding exists iff $\la< \frac 65$.  \SSS

\NI {\rm (ii)}
The question of whether
$\la E(1,5)$ embeds in $\oE(2,3)$ is equivalent to the packing problem with $k=8$ and $\un w=\frac 13(1,1,1,\la,\la,\la,\la,\la)$. Hence the embedding exists iff $\la<\frac{12}{11} $.
\end{prop}

These packing problems arose in Tolman's attempt to describe
 all $6$-dimensional Hamiltonian $S^1$-manifolds $M$ with $H^2(M)$ of rank $1$.  In \cite{Me} we use the ideas of the present paper to 
 construct the two new manifolds among her list of four possibilities, thus completing her classification.

\begin{rmk}\rm (i)  The question of which weights $\un w$ correspond to 
a given packing problem is intimately related to standard (rather than Hirzebruch--Jung) continued fraction expansions: see 
Remarks \ref{rmk:cf} and \ref{rmk:sing}.\MS 

\NI (ii)
Our approach also gives a great deal of information
about the function 
$$
c:[1,\infty)\to [1,\infty),\quad c(a) = \inf\{ \mu: E(1,a)\se B(\mu)\},
$$
studied in Schlenk \cite{Sch} and  in \cite[\S4]{CHLS}.  This will  be the subject of  McDuff--Schlenk \cite{McS}. 
\end{rmk}

\NI {\bf The nature of the ball packing problem.}\,\,
We now explain the results of \cite{MP,M,B,LL} that first convert the 
symplectic packing problem for balls into a question of understanding 
the symplectic cone  of the $k$-fold blow up $X_k$ of 
$\C P^2$, and then explain the structure of that cone.  
The {\it symplectic cone} of an oriented manifold is the set of cohomology classes of $M$ with symplectic representatives compatible with the given orientation.  By Li--Liu \cite{LL},
when $M = X_k$ this is a disjoint union of  
(connected) subcones, each consisting of forms $\om$ with a given canonical class $K = -c_1(M,\om)$.  Moreover,  the group of diffeomorphisms of $X_k$ acts transitively on these subcones.
Fix $K: = -3L +\sum E_i$, where $L: = [\C P^1]$ 
and $E_1,\dots E_k$ are the exceptional divisors, 
and define
$\Cc_K(X_k)\subset H^2(X_k,\R)$  
to be the set of classes represented by symplectic forms
with canonical class $K$. 
Also, denote the Poincar\'e duals of $L,E_i$ by  $\ell, e_i$.

\begin{prop} [McDuff--Polterovich \cite{MP}]\labell{prop:MP}
It is possible to embed $k$ balls with weights $\un w$  
into $\C P^2(1)$  or $\oB^4(1)$ if and only if the  
class $a_{\un w}: = \ell - \sum w_i e_i$ lies in $\Cc_K(X_k)$.
\end{prop}

Therefore the problem is equivalent  to understanding $\Cc_K(X_k)$.
Though not formulated in these terms, Lemma 1.1 of \cite{M} in essence
describes the closure of $\Cc_K(X_k)$.  (The argument is sketched below.)
That paper solved the uniqueness question for symplectic packings.
In \cite{B,BE},  Biran recast this lemma  in terms of the symplectic cone,
using it to solve the
 existence question for  packings by equal balls.  However the cone itself,
 rather than its closure, was first described in Li--Liu \cite{LL}. 
   
They consider the set 
$\Ee_K(X_k)\subset H_2(X_k;\Z)$  of classes $E$ 
with $K\cdot E = 1, E^2=-1$ that can be represented by smoothly 
embedded $-1$ spheres. By Theorems B and C in Li--Li \cite{LL2}, 
given any symplectic form
$\om$ with canonical class $K$, each $E\in \Ee_K(X_k)$ can be 
represented by an $\om$-symplectically embedded $-1$ sphere.
Thus the above set  $\Ee_K(X_k)$ is the same as that
used in \cite{MP,M,B}.

\begin{prop} [Li-Liu, \cite{LL}] \labell{prop:M}
$$
\Cc_K(X_k) = \{a\in H^2(X_k): a^2>0, \,a(E)>0 \mbox{ for all } 
 E\in \Ee_K(X_k)\}.
$$
\end{prop}
\NI {\it Sketch of proof.} 
Let $\om$ be any symplectic form on $X_k$ with canonical class $K$.  
It follows from the work of Kronheimer and Mrowka  
\cite{KM} on wall crossing 
in Seiberg--Witten theory that for all 
$a\in H^2(X_k;\Q)$ with $a^2>0$ 
the class $qa$ has nontrivial Seiberg--Witten invariant for all 
sufficiently large integers $q$. Therefore, by work of Taubes \cite{Tau},
the Poincar\'e dual $PD(qa)$ has a $J$-holomorphic representative for 
every $\om$-tame $J$. Moreover, 
if $a(E)\ge 0$ for all $E\in \Ee$
and $J$ is generic this representative
is a connected and embedded submanifold.  
By inflating along 
this submanifold one can construct 
a family of symplectic forms $\om_t$ with $\om_0=\om$ and 
such that $[\om_t]$ converges to $a$ as $t\to \infty$.
 Therefore, the class  $a$ 
is in the closure of $\Cc_K(X_k)$; see \cite[Lemma~1.1]{M}.
A more careful version of this argument
shows that  $a$ 
 must in fact be  in $\Cc_K(X_k)$ provided that
  $a(E)>0$ for all $E$; see \cite[\S4]{LL}.
\QED

%
%

The arguments in \S2 below show explicitly how to use these ideas to 
construct full packings  by the ellipsoids $E(1,k)$ for $k\in \Nn$ and hence
by equal balls.
 The general existence question is fully 
understood when $k<9$ since in that case $\Ee_K$ is finite and is 
easily enumerated.  However,  it is 
not always so easy to answer specific problems when $k\ge 9$ 
since $\Ee_K$ is more complicated. We shall return 
to this question in\cite{McS}.

\begin{rmk}\rm The above results imply  there are two obstructions 
to the existence of  an embedding from one ellipsoid to another.  
If $a_{\un w}$ is the class in $X_k$ 
of the corresponding packing problem, one needs:\MS

 (i) $a_{\un w}^2>0$, and\SSS
 
  (ii) $a_{\un w}(E)>0$ for all $E\in \Ee_K(X_k)$.\MS
  
\NI
The first condition is equivalent to the volume obstruction, while 
the second is a generalization of the condition used by Gromov \cite{G} 
to find a packing obstruction when $k=2,5$. It is given by 
 the rigid $J$-holomorphic spheres in $X_k$, 
and hence should appear as a genus zero holomorphic trajectory 
in other contexts. However, note that although the classes $E\in \Ee_K$ 
are rigid in the blow up, they correspond in $\C P^2$ to curves 
that satisfy some constraints, e.g. they might have to
 go through a certain number of points with given multiplicities.   
Moreover, in general these constraints may not be purely homological, 
but may involve descendents: cf. the blow down formulas in \cite{HLR}.
Thus, for example, if one tried to understand the obstructions to 
embedding $E(m,n)$ into $E(m',n')$  
by looking at  the
properties of the induced cobordism between their boundaries
 then these curves should appear in one of the higher dimensional 
(but genus zero) SFT moduli spaces and hence should be visible, 
provided that one works in a  context that takes these higher 
dimensional spaces into account. 
 \end{rmk}

\NI {\bf Organization of the paper.} \S2 
considers the problem of embedding ellipsoids of the form $E(1,k)$ 
into balls, where $k\in \N$.  The case $k=d^2$ is particularly simple
and is treated first. Theorem \ref{thm:1} is proved in \S2.2. In order to deal with general integral ellipsoids one must understand exactly how to approximate them by chains of spheres.
This is the subject of \S3.1. Theorem \ref{thm:ell}, 
Corollary~\ref{cor:conn} and Proposition~\ref{prop:tol} 
are proved in \S3.2.\MS

\NI {\bf Acknowedgements.} I was inspired to think about this problem by discussions with Hofer and  Guth, who in \cite{Gu} recently solved (in the negative) Hofer's question about the existence of higher dimensional capacities. I also thank Tolman for giving me an advance copy of her paper \cite{Tol},  Schlenk for his encouragement and useful comments, and the referree for pointing out various small inaccuracies.

\section{Embedding  $E(1,k)$ into a ball.}

We first give a direct geometric construction for embedding $E(1,d^2)$ into a ball.  We then prove Theorem~\ref{thm:1}.

\subsection{The case $k=d^2$.}

This section proves the following result.

\begin{prop}\labell{prop:d}  $E(1,d^2)$ fully fills $B(d)$.
\end{prop}

Denote by $\De(m,n)$ the triangle with outward conormals 
$(-1,0), (0,-1), (m,n)$ and vertices at $(0,0), (n,0),(0,m)$; see Figure \ref{fig:1}. 
If $m,n$ are integers then the open ellipsoid 
$\oE(m,n)\subset \C^2$ is invariant under the obvious $T^2$ action and is taken by the moment map onto the \lq\lq interior" 
$\ooDe(m,n)$ of $\De(m,n)$, which we define to be the complement of its slanted edge. 
 The closed triangle $\De(m,n)$ is the moment polytope of a weighted projective space, but we sometimes think of it as
  the moment polytope of $E(m,n)$ since it is the 
  image of $E(m,n)$ under the moment map.  
 
 Recall that if $(m_1,n_1), (m_2,n_2)$ are the outward conormals to
two successive edges (ordered anticlockwise) of a moment polytope
in $\R^2$ then their intersection is the image of an orbifold point of order $k$  iff 
\begin{eqnarray}\labell{eq:tor}
\left|\!\!\begin{array}{cc} m_1&n_1\\m_2&n_2\end{array}\!\!\right| = k.
\end{eqnarray}
 In particular, 
it is smooth iff $k=1$. We shall call a moment polytope 
{\it smooth} if all its vertices are smooth.  (For basic information on toric geometry in the present context see Symington \cite{Sym} and Traynor \cite{T}.)

 \begin{figure}[htbp] 
    \centering
    \includegraphics[width=2in]{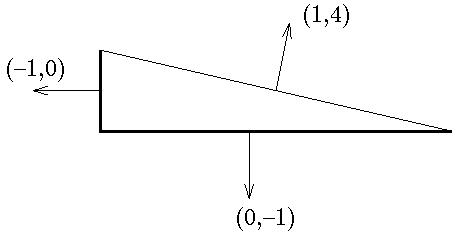} 
    \caption{A toric model of the open ellipsoid $\ooE(1,4)$; it
      does not include the slanted edge.}
    \label{fig:1}
 \end{figure}

\begin{figure}[htbp] 
   \centering
   \includegraphics[width=1.5in]{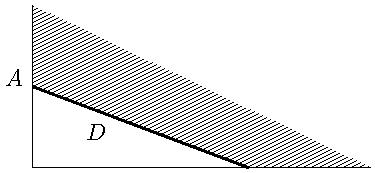} 
   \caption{Toric model of an ellipsoidal blow up}
   \label{fig:2}
\end{figure}

Just as in the case of balls, if $\la E(m,n)$ embeds in $(M,\om)$ then one can cut it out and collapse the boundary along the characteristic flow to obtain
an orbifold $\ov M$ with an exceptional divisor $D$ (the heavy line in Figure \ref{fig:2}).    In general,  the two new vertices
are singular points of $\ov M$.
 However, in the case of $\la E(1, k)$ there is just one singular point on $D$ at $A$.  For more details, see \cite{God}.

The main idea of the current note is that instead of working with the  orbifold $\ov M$ we can cut away more (i.e. blow up further) in order to get a smooth manifold.   
The case $(m,n) = (1,k)$ is particularly simple; we simply need to blow up $k-1$ more times. In the toric picture this amounts to cutting the polytope along lines with conormals
$(-1,-1), (-1,-2), (-1,-3),\dots, (-1, -k)$. The curve $D$ is then
transformed into an exceptional curve in a smooth manifold.
As we see below, it is possible to carry out  this process omitting all mention of orbifolds. We explain it first in the case $k=4$.

\begin{figure}[htbp] 
    \centering
    \includegraphics[width=4in]{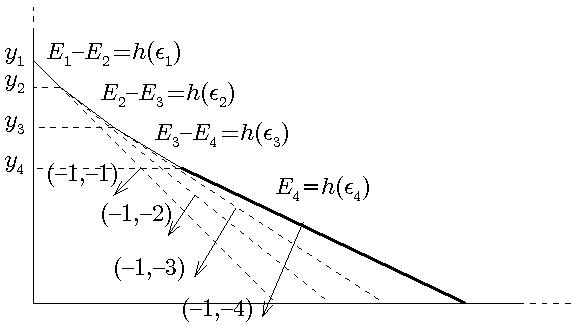} 
    \caption{The three blow ups needed to resolve $A$ in the case $E(1,4)$. The four edges are named $\ve_i$ and their homology classes $h(\ve_i)$.}
    \label{fig:3}
 \end{figure}

When $k=4$, this blow up process gives a configuration of $4$ spheres intersecting transversally,
the $-1$-sphere $C_4$ in class  $E_4$, and three $-2$-spheres $C_i, i=1,2,3,$ each in class $ E_i-E_{i+1}$: see Figure \ref{fig:3}. 
If we start off  in a large ball or in a large $\C P^2$, we can construct such a configuration.  Namely, start with the first quadrant; form $E_1$ by a cut with conormal $(-1,-1)$ and size  
$\la+\de_1$. (The size 
of the cut is given by the affine length\footnote
{
The {\it affine length} $\al(\ve)$ of an edge $\ve$ of a moment polytope can be measured as follows.  Take any  affine transformation $\Phi$ of $\R^2$
that preserves the integer lattice and is such that $\Phi(\ve)$ lies 
along one of the axes, and then measure the Euclidean length 
of $\Phi(\ve)$. Thus if $\ve$ has rational slope and  endpoints on the integer lattice, $\al(\ve) =k+1$ where $k$ is the number of  points 
of the integer lattice in the interior of $\ve$. Note also that the
divisor that is taken to $\ve$  by the moment map has symplectic area $\al(\ve)$. 
}
 of the resulting exceptional divisor, and hence in this case is given by $y_1$, the height of the 
point where the cut meets the $y$ axis.)  Then cut off most of $E_1$
by a cut with conormal $(-1,-2)$ and size $y_2=\la+\de_2< y_1$.   Make one more cut along $E_3$ of size $y_3=\la+\de_3<y_2$.  Then $E_4$ is what is left of $D$. It is easy to check that it has size
 $y_4= 4\la - (\la+\de_1)- (\la+\de_2)- (\la+\de_3) =: \la -\de_4$, where $\de_4=\sum_{i\le3} \de_i$.

Let us denote by $\Nn(\Hat \Cc_4(\la,\de))$  a small open neighborhood
in the toric model
of the final configuration $\Hat \Cc_4(\la,\de)$ of $4$ spheres.  Thus $\Hat \Cc_4(\la,\de)$ is
a configuration
 of $4$ symplectic spheres $C_1,\dots,C_4$ where $C_i$ intersects 
 $C_{i+1}$ transversally for $i=1,2,3$
 (and there are no other intersection points),  where $C_4$ 
has self-intersection $-1$ and size $\la -\de_4$ and where the other
 $C_i$  have self-intersection $-2$ and (positive) 
sizes $\de_1-\de_2, \de_2-\de_3$ and $
\de_3+\de_4$. In the following, we assume that $\de_1>\de_2>\de_3>0$ 
 and that $\de_4 = \de_1+\de_2+\de_3$.
 
 \begin{lemma}\labell{le:nb}
Any embedding of  $\Hat \Cc_4(\la,\de)$ into $(M,\om)$ is isotopic 
to one that extends to an embedding of  $\Nn(\Hat \Cc_4(\la,\de))$.
 \end{lemma}
\begin{proof} Note that in any smooth
toric manifold the spheres represented by $2$ edges meet orthogonally since we can put them on the axes by an affine transformation.
This may not be the case for the given embedding of 
$\Hat \Cc_4(\la,\de)$.  However, one can slightly perturb this embedding so that the different spheres do meet  orthogonally. 
(A similar point occurs in the proof of \cite[Thm~9.4.7]{MS}. See also
\cite[Ex.~9.4.8]{MS}.) It then follows from the symplectic neighborhood theorem that the embedding extends to $\Nn(\Hat \Cc_4(\la,\de))$.
\end{proof}

  Denote by $X_4(\mu;\la,\de)$ the $4$-point blow up of the projective plane in which the line has symplectic area $\mu$ and we have blown up $4$ times by the amounts
  $\la+\de_i, i = 1,2,3$ and $\la-\sum\de_i$. If $\la <1$ we can choose the $\de_i$ so that each $\la+\de_i < 1$.
    Because $\C P^2$ can be fully filled by $4$ balls of equal size, 
we can  therefore construct $X_4(2;\la,\de)$ for any $\la<1$ 
and sufficiently small
     $\de: = (\de_1,\de_2,\de_3)$ as above.

\begin{lemma}\labell{le:1}
   $\la E(1,4)$ embeds into the interior of $B(2)$ iff
$\Hat \Cc_4(\la,\de)$ embeds into the complement of a line in $X_4(2;\la,\de)$ for some small $\de$.
\end{lemma}
\begin{proof}  If $\la E(1,4)$ embeds then this embedding
extends to $(\la+\ka)E(1,4)$ for some small $\ka>0$.  Hence
$\la E(1,4)$ has a standard neighborhood with a toric structure as above.
Complete $B(2)$ to $\C P^2(2)$ and then blow up as explained above to get 
an embedding of $\Hat \Cc_4(\la,\de)$  into the complement of a line in $X_4(2;\la,\de)$. 

Conversely, suppose that $\Hat \Cc_4(\la,\de)$ embeds
into the complement of a line in $X_4(2;\la,\de)$. 
Then, by Lemma \ref{le:nb} we may suppose that some standard neighborhood $\Nn(\Hat \Cc_4(\la,\de))$ also embeds.  By 
Symington's discussion in \cite{Sym} 
of toric models for the rational blow down, 
we may then \lq\lq blow down"
$\Hat \Cc_4(\la,\de)$, i.e. perform a symplectic surgery along this nonsmooth divisor that adds the singular piece that was cut out when resolving the singularity. The resulting blow down manifold $(M,\om)$ contains a symplectically embedded copy of $\C P^1$ of size $2$ and a disjoint copy of 
$\la E(1,4)$.  Moreover, it is diffeomorphic to $\C P^2$.  Therefore, by Gromov's uniqueness result for symplectic forms on $\C P^2$ (cf \cite[Ch.~9]{MS}), 
$(M,\om)$ can be identified with $\C P^2(2)$, so that $M\less \C P^1$ is the interior of the 
standard ball $B(2)$.  This completes the proof.
\end{proof}

\begin{rmk}\label{rmk:rmk1}\rm
There is an obvious analog of this result for any $k$.  An
appropriate definition of $\Hat\Cc_k(\la,\de)$ is explained in the 
proof of Theorem 1.1 at the end of this section.\end{rmk}

Therefore we just need to embed 
$\Hat \Cc_4(\la,\de)$ into the complement of a line in $X_4(2;\la,\de)$.  This is possible for small $\la$.  We then will use the inflation process to increase $\la$. In this case, the construction can be done entirely explicitly: there is  no need to use Seiberg--Witten
theory.\MS

\NI {\bf Proof of Proposition~\ref{prop:d}.}
\MS

\NI {\bf Step 1:} {\it Explicit embedding of $\Hat \Cc_4(\la,\de)$.}\,
Start with $\C P^2(2)$.  Let $Q\subset \C P^2$ be a smooth conic, $L$ a line and $p_1$ a point
on $Q$ but not $L$. Blow up at $p_1$ with size $\la+\de_1$ where $\la$ is small and $\de_1$ is tiny.  Let  $p_2$ be the intersection of
 the exceptional divisor $E_1$ with the proper transform $Q_1$ of $Q$
 and blow up at $p_2$ with size $\la+\de_2$ to get an exceptional divisor $E_2$.   Now repeat this twice more, blowing up at 
 $$
 p_2\in E_2\cap Q_2,\quad Q_2: = \mbox{proper transform of }Q_1
 $$
  by $\la+\de_3$ to get exceptional divisor $E_3$ 
and finally blowing up at 
$$
p_3\in E_3\cap Q_3,\quad Q_3: = \mbox{proper transform of }Q_2
$$
 by $\la -\de_4$ (where $\de_4: = \sum_{i=1}^3\de_i$) to get $C_4$ in the $4$-fold blow up $X_4$. For $i<4$ 
 denote by $C_i$ the proper transform of $E_i$ in the next blow up.
Then
 $[C_i] = E_i-E_{i+1}$ for $i\le 3$ and $[C_4] = E_4$.
Thus we have constructed a copy of the configuration
$\Hat \Cc_4(\la,\de)$ in $X_4(2;\la,\de)$.  Denote the symplectic form
on $X: = X_4(2;\la,\de)$ by $\om_0$.

Note that if $\la$ is sufficiently small  we may assume that none of these blowups affect $L$.  The conic $Q$ becomes
a curve $Q_0$  in class $2L-E_1-E_2-E_3-E_4$ and so has area
$4-4\la$.  By construction $Q_0$ meets $C_4$ once but not $C_i, i<4$.
Moreover, 
$$
\begin{array}{lll}
\int_{C_1}\om_0 = \de_1-\de_{2}, &
\int_{C_2}\om_0 = \de_2-\de_{3}, &
\int_{C_3}\om_0 = \de_3+\de_{4}\\
\int_{C_4}\om_0 = \la-\de_4,&
\int_{Q_0}\om_0 = 4-4\la,&
\int_L\om_0 = 2.
\end{array}
$$
\MS

\NI {\bf Step 2:} {\it The inflation process.} 
 We will inflate $(X,\om_0)$ along $Q_0$.
Note that $Q_0\cdot Q_0 = 0$.  Therefore we may identify  a
neighborhood $\Nn(Q_0)$ with the product  $(Q_0\times D^2,\om_Q\times \al)$ where $\al$ is some area form on $D^2$.
By perturbing  
$C_4$ and the line $L$  and then  shrinking $\Nn(Q_0)$ if necessary, 
we may assume that 
\SSS

$C_i\cap \Nn(Q_0) = \emptyset$ for $i\le 3$;

$C_4\cap \Nn(Q_0)$  is a flat disc $pt\times D^2$;

$L\cap \Nn(Q_0)$  is the union of two disjoint flat discs $pt\times D^2$.
\SSS

\NI
Let $\be$ be a nonnegative form on the two disc $D^2$ with support in its interior and
$\int_{D^2}\be = 1$.
Define $\om_t: = \om_Q\times (\al + t\be)$  in $\Nn(Q_0)$
and equal to the original symplectic form  $\om_0$ outside $\Nn(Q_0)$.
This is clearly symplectic everywhere.
Then the integral of $\om_t$ over $Q_0$ and the $C_i,i\le 3$, is constant while
$$
\int_{C_4} \om_t = \la - \de_4 + t,\quad 
\int_{L} \om_t = 2 + 2t.
$$
Given any $\la_0<1$, 
choose $T$ so that  $\la': =\frac{\la - \de_4 + T}{1+T} \ge \la_0$
and set $\tau: = \om_T/(1+T)$.
Define $\de_i' = \de_i/(1+T)$. By the uniqueness of symplectic forms on
the blow ups of $\C P^2$ (cf. \cite{M}) we may identify
$(X,\tau)$ with $X_4(2;\la',\de')$, and may also identify the configuration
with  $\Hat \Cc_4(\la',\de') $.  Since the configuration is disjoint from a line by construction, we can now deduce  the proposition from Lemma~\ref{le:1}.
\MS

\NI {\bf Step 3:} {\it Completion of the proof.}
This argument extends immediately to the case $d>2$:
simply replace $Q$ in the 
above construction by a degree $d$ curve in $\C P^2(d)$, and blow up $d^2$ times. 
\QED

\begin{rmk}\rm  Because the above argument 
 uses symplectic inflation, it does  not give a completely explicit
 geometric  construction of the embedding.  It is not clear whether such a construction exists. (The construction in Opshtein \cite{Op} gives an embedding into $\C P^2$.)
 The best results obtained so far by explicit construction are those of Schlenk, who used a multiple symplectic folding technique to show that $E(1,4)$ embeds in $B(\mu)$ for $\mu$ approximately equal to $2.7$, see ~\cite[Ch.~3.3]{Sch}.
\end{rmk}

\subsection{The general case.}

We now prove Theorem~\ref{thm:1}. 

\begin{figure}[htbp] 
   \centering
   \includegraphics[width=2.5in]{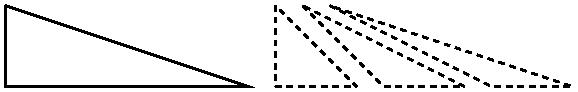} 
   \caption{Three copies of $\overset{\scriptscriptstyle\circ}B(1)$ embed in $E(1,3)$.}
   \label{fig:5}
\end{figure}

\begin{lemma}\labell{le:open}  $E(1,k)$ contains the disjoint union of $k$ open balls $\oB(1)$.
\end{lemma}
\begin{proof}  Denote by $\breve\De\subset \R^2$  the open triangle
with vertices $(0,0),  (1,0) $ and $(0,1)$. 
As is clear from Figure \ref{fig:5}, there are integral affine transformations $A_1=id, A_2,\dots, A_k$ of $\R^2$ such that 
the 
$k$ triangles $A_1(\breve\De),\dots,A_k(\breve\De)$ embed in the moment polytope of $E(1,k)$.  
On the other hand
$\breve\De$ is the moment polytope of the open subset 
$$
U = 
\{(z_1,z_2)\in \C^2\,|\, z_i\ne 0, |z_1^2|+|z_2|^2<1\}\subset B(1),
$$
which contains an embedded image of the open ball $\oB(1)$  by Traynor \cite{T}.
\end{proof}

\NI {\bf Proof of Theorem~\ref{thm:1}.} If $E(1,k)$ embeds in $\oB(\mu)$ then $(1+\ka) E(1,k)$ also embeds for some $\ka>0$.  
But Lemma~\ref{le:open} implies that $(1+\ka) E(1,k)$  contains $k$ disjoint closed balls $B(1)$.  Hence so does $\oB(\mu)$.  

Conversely, if $k$ copies of $B(1)$ embed in $\oB(\mu)$, choose   $\mu_0<\mu$ so that their image is contained in $\oB(\mu_0)$.
By Proposition~\ref{prop:MP}
the class $a: = \mu_0\ell - \sum_{i=1}^k e_i\in H^2(X_k)$ has a  symplectic representative with canonical class $K$ and so lies in $\Cc_K(X_k)$.  Hence, as in Proposition~\ref{prop:M},
 $a(E)>0$ for all $E\in \Ee(X_k)$.
Without loss of generality, we may suppose that $\mu_0\in \Q$ and then choose $q$ so large that $qa\in H^2(X_k;\Z)$ and $PD(qa)$ has nontrivial Gromov invariant.

Now let $\Hat\Cc_k(\la,\de)$ be a configuration of $k$ symplectic spheres
$C_1,\dots,C_k$ such that\MS

\NI (a) $C_i\cap C_j\ne \emptyset$ iff $|i-j|<2$,\SSS

\NI (b)  $C_i^2 = -2, i<k$, and $C_k^2 = -1$, and  \SSS
 
\NI (c)  area$(C_i) =  \de_i-\de_{i+1}$ if $1\le i< k-1$,
area$(C_{k-1}) = \de_{k-1} + \de_k$, and area$(C_k) = \la - \de_k$,
where 
$\de_1>\dots >\de_{k-1}>0$ and
$\de_k = \sum_{i<k}\de_i$.\MS

If $\la$ and the $\de_i$ are sufficiently small, the 
blow up construction in \S2.1 shows that there is a symplectic form 
$\om_0$  on $X_k$ 
with $\int_L\om_0=1$ and such that the configuration
$\Hat\Cc_k(\la,\de)$ embeds in the complement of a line $L$ 
in $(X_k, \om_0)$ in such a way that 
$C_i$ lies in class $E_i-E_{i+1}$ for $i<k$ and  $[C_k] = E_k$. 
(Observe here that the class $[\om_0]$ is determined by  
$\la$ and the $\de_i$: it will not equal $a$.)  
Moreover, $\Hat\Cc_k(\la,\de)$  has a toric neighborhood $\Nn(\Hat\Cc_k(\la,\de))$ constructed just as in the case $k=4$, and by Lemma~\ref{le:nb} we may suppose that $\Nn(\Hat\Cc_k(\la,\de))$ also embeds
disjointly from $L$.  
For simplicity, we shall identify $\Nn(\Hat\Cc_k(\la,\de))$ with its image in $(X_k,\om_0)$. 

Denote by $\Jj_\Nn$ the set of  $\om$-tame $J$ for which 
$\Hat\Cc_k(\la,\de)$ and the line $L$ are holomorphic.
Because $PD(qa)$ has nontrivial Gromov invariant and $qa(E)>0$ for all 
$E\in \Ee_K$, $PD(qa)$ is represented by
an embedded $J$-holomorphic curve $Q$ for every $\om$-tame $J$ that is 
sufficiently generic.  Since no smooth curve in class $PD(qa)$ is represented entirely in $\Hat\Cc_k(\la,\de)\cup L$, it follows from 
\cite[Ch~3]{MS} that we can take $J$ to be a generic element of $\Jj_\Nn$.   
 Then, by positivity of intersections, the fact that $a(C_i) = 0, i<k,$ implies that $Q\cap C_i = 0$. Further, we may perturb $Q$ so that it   intersects $C_k$   transversally $q$ times and $L$  transversally
 $q\mu_0$ times.

Now inflate along $Q$. This construction gives 
a family $\om_t, t\ge 0,$
of symplectic forms on $X_k$ lying in class $[\om_0] + tqa$ that equal $\om_0$ outside a small neighborhood of $Q$ and restrict
 on $C_k$ (resp.  $L$) to
a symplectic form of area $\om_0(C_k) + qt=\la-\de_k+qt$ (resp. $1+qt\mu_0$). Thus $\Nn(\Hat\Cc_k(\la-\de_k+qt,\de))$
embeds in $(X_k,\om_t)$. Therefore,  by Lemma~\ref{le:1} (see also 
Remark~\ref{rmk:rmk1}), 
$(\la+qt)E(1, k)$ embeds in $\oB(1+qt\mu_0)$ for all $t>0$.
Hence
$$
E(1,k)\;\;\se \;\;\ts{{\frac{1}{\la+qt}}\,\oB(1+qt\mu_0)}.
$$
Since  $\mu_0<\mu$, $\frac{1+qt\mu_0}{\la+qt}<\mu$ for large $t$.
Hence the result.
 \QED

\section{Embedding  ellipsoids into ellipsoids.}

We first show  how to find the weights of the ball embedding problem that is equivalent to a given ellipsoidal embedding problem. Theorem \ref{thm:ell}
 is proved in \S\ref{ss:thm}.

\subsection{Toric approximations to ellipsoids.}\labell{ss:approx}

We begin by discussing inner and outer approximations.
Throughout, $(m,n)$ are mutually prime and $0<m \le n$.
As in the case $(m,n) = (1,k)$, it is possible to approximate the moment polytope
$\De(m,n)$  of  $E(m,n)$
 by a 
smooth polytope $\De'$  by blowing it up appropriately.  Since every moment polytope is  a blow up of $\De(1,1)$
 (up to scaling and integral affine transformation), one can equivalently start with $\De(n,n)$ and by a sequence of blow ups  arrive at a polytope $\De'$ lying inside $\De(m,n)$ and with one conormal equal to $(m,n)$. One can then adjust the side lengths of this polytope to make  its edge $\vareps_N$ with conormal $(m,n)$ coincide with part of the corresponding edge of $\De(m,n)$ and also
 the region $\De(m,n)\less \De'$ arbitrarily small (in area).
We will call such $\De'$ an {\it inner approximation} to $\De(m,n)$;
see Figure ~\ref{fig:6}(ii).  This is the relevant approximation when
 $E(m,n)$ is the target of the embedding.
 If $E(m,n)$ is the source, then as in \S2 one should look for
{\it outer approximations} $\De''\supset \De(m,n)$ such that 
$\De''\less\De(m,n)$ is small.  These are essentially the same as
inner approximations to $\De(n',n')\less \ooDe(m,n)$ for  $n'>n$: see Remark~\ref{rmk:cf} (i).

\begin{figure}[htbp] 
   \centering
   \includegraphics[width=5in]{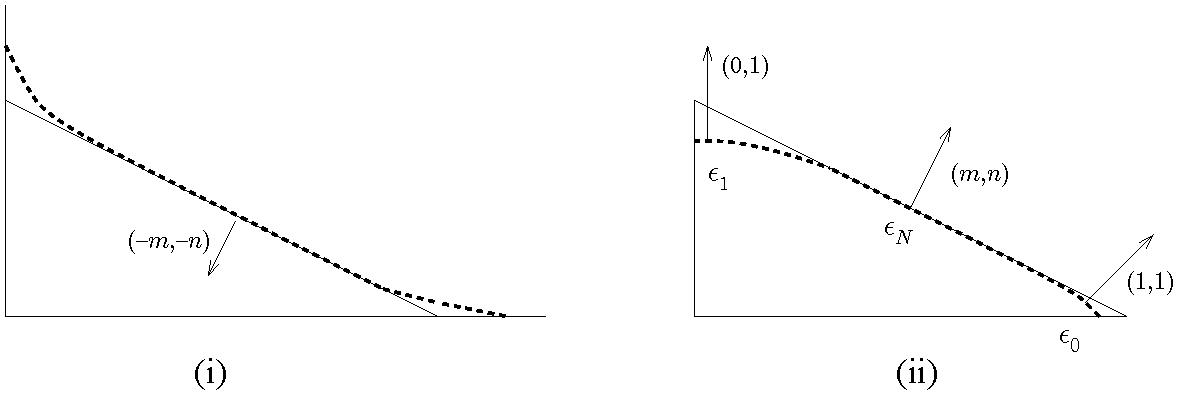} 
   \caption{(i) is an outer approximation to $\De(m,n)$ while
    (ii) is an inner approximation}
   \label{fig:6}
\end{figure}

For clarity, let us first concentrate on inner approximations.
There are many possible choices of approximation. However, as we show below, there is a unique  minimal sequence of blow ups of $\De(n,n)$ with exceptional divisors $E_1, E_2,\dots,E_N$ such that the conormal to the edge $\vareps_N$ given by the last blow up is $(m,n)$ and so that none of the other new edges have 
self intersection $-1$.\footnote{
If one of the other edges did have self-intersection $-1$, it would have to be disjoint from $\vareps_N$ and hence could be blown down.}
 Let us denote by $\vareps_0$ the edge with conormal $(1,1)$ and by $\vareps_i$ the edge created by the $i$th blow up. Further, denote by \MS
 
 $h(\vareps_i)$ the homology class of the edge in the $N$-fold blow up $X_N$;\SSS
 
$\al(\ve_i)$ the affine length of $\ve_i$;\SSS

$\nu(\ve_i)$ the conormal of $\ve_i$.\MS

 Then our conditions imply:
 \MS
 
 \NI $\bullet$  $h(\vareps_N) = E_N$;\SSS
 
\NI $\bullet$ for $1\le i<N$, $h(\vareps_i)=E_i - E_{i_1} -\dots - E_{i_k}$ for suitable $i<i_1<\dots<i_k$ and $k>0$. \SSS

 \NI $\bullet$  $h(\vareps_0) = L - E_{0_1} -\dots - E_{0_k}$ for 
 suitable indices $0_1<\dots< 0_k$, where $L$ is the class of the line in $\C P^2$. \MS

We shall first discuss the blow up process outlined above and then
consider how to choose the lengths $\al(\ve_i)$.
\MS

\NI {\it Construction of the minimal blow up sequence for $(m,n)$.}
The blow up process replaces the intersection of two adjacent edges with conormals $(p,q), (p',q')$
by a new edge with conormal $(p'',q'')= (p+p', q+q')$.  Thus the
three fractions $p/q$ are related by the identity
$$
\frac {p''}{q''} = \frac{p+p'}{q+q'}.
$$
Hence they are adjacent terms in the {\it Farey sequence} $\Ff_K$, $K: = q+q'$.  (Recall from Hardy--Wright \cite[Ch.III]{HW} that
$\Ff_K$ is the finite sequence obtained by arranging the fractions 
$p/q$,  where $0\le p\le q$ are relatively prime and $q\le K$,
 in order of increasing magnitude.)  It is well known that $\Ff_K$ can be constructed by starting with the fractions $\frac 01$ and $\frac 11$ and then repeatedly inserting the fractions $\frac{p+p'}{q+q'}$ with $q+q'\le K$ between any two neighbors $\frac {p}{q},\frac {p'}{q'}$.  Since this precisely corresponds to the blow up procedure, it follows that one always can find a sequence of blow ups 
that starts from $\vareps_0 = (1,1)$ and $\vareps_1=(0,1)$ and ends up with an edge $\vareps_N$ with conormal $(m,n)$.  Moreover, to find a minimal sequence one should include $\frac{p+p'}{q+q'}$ in the sequence 
only if $\frac mn$ lies between
the points $\frac {p}{q}$ and $\frac {p'}{q'}$. We denote the corresponding connected chain of adjacent edges by $\Ee(m,n)$.  This chain, when ordered by increasing $\frac mn$, 
starts with $\vareps_1 = (0,1)$, includes an edge $\vareps_N$ with conormal $(m,n)$ and ends with $\vareps_0=(1,1)$.  (Note that these edges are numbered according to the order in which 
 the blow ups are performed, not 
by adjacency.) 

For any pair $(m,n)$ with $m<n$ we can construct a convex chain of edges 
$\Ee(m,n)$ of this form, such that all edges except for the last one $\vareps_N$ are very short and so that $\vareps_N$ is almost all of the slanted edge of $\De(m,n)$.  To do this, first place $\eps_0$ 
so that it meets the $x$ axis at $x_0: = (n-\de_0,0)$ for small $\de_0>0$ and make the first cut $\vareps_1$ so that it meets the $y$ axis at $y_1:=(0,m-\de_0-\de_1)$ for small $\de_1>0$.  
Then perform all subsequent blow ups so that all the edges
 $\vareps_i$ have positive length and so that $\vareps_N$ 
coincides with part of the line $mx+ny=mn$: see Figure \ref{fig:9}.

\begin{figure}[htbp] 
   \centering
   \includegraphics[width=3in]{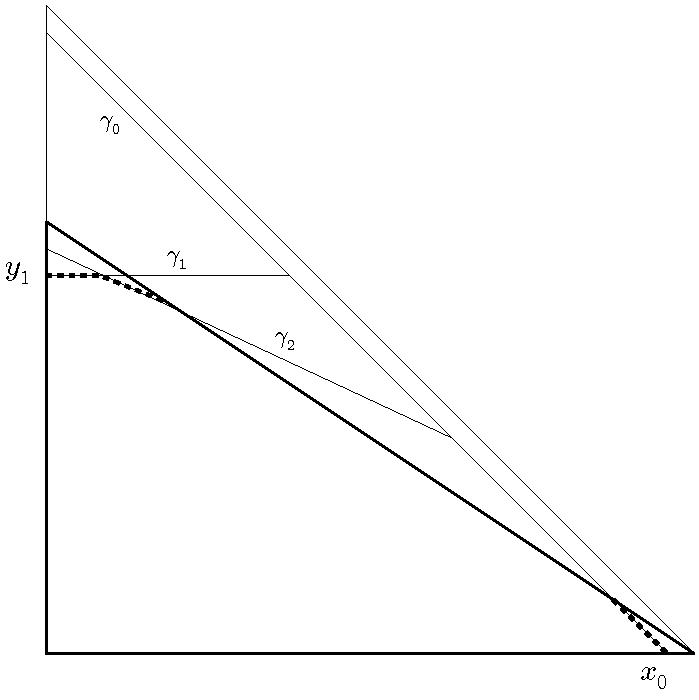} 
   \caption{An inner approximation to $\De(2,3)$.  Here $x_0 = n-\de_0$ and $y_1 = m-\de_0-\de_1$ are the points
    where the edges $\vareps_0,\vareps_1$  meet the  axes.  The cuts $\ga_i$ are also labelled.}
   \label{fig:9}
\end{figure}

Note that if the conormal $(p,q)$ of $\vareps_i$ has $\frac pq< \frac mn$ then the line of the
corresponding cut must meet the $y$ axis at some point $(0,y_i)$ with $y_1<y_i<m$, while if  $\frac pq> \frac mn$ it meets the $x$-axis 
at $(x_i,0)$ where $x_0<x_i<n$.
  It follows
that the edges for $i<N$ must be very short, while $\vareps_N$ is almost all the slanted edge of $\De(m,n)$.   

This chain of edges $\Ee(m,n)$ lies in $\De(m,n)$ and bounds a smooth subpolytope $R_N$ of $\De(m,n)$.  The above sequence of blow ups 
gives a way to construct 
the polytope $R_N$  
from $\De(n-\de_0,n-\de_0)$ by performing a sequence of blow ups
with conormals $\nu(\vareps_1),\dots,\nu(\vareps_N)$.
We shall denote by $R_i, i\ge 1$ the region obtained after the $i$th blow up and by $\ga_i$ the new edge of $R_i$ formed by the $i$th blow up.  Thus $\ga_i$ is an extension of $\vareps_i$. Notice that $R_{i} = R_{i-1}\less \De_i$, where the triangle $\De_i$ is equivalent under the action of $\SL(2,\Z)$ to $\mu_i\De(1,1)$ and where $\mu_i$ is the affine length of the cut edge $\ga_{i}$.  It follows from this construction that the homology classes $h(\vareps_i)$ of the edges $\vareps_i$ are given by
\begin{equation}\labell{eq:hi}
h(\vareps_i) = E_i - \sum_{j\in S_i}E_j,\mbox{ where }
S_i = \{j>i:\mbox{ the edge } \ga_j \mbox{ intersects }\ga_i\}.
\end{equation}

%


%

%


As well as this minimal blow up sequence, we shall need an integral  homology class $V_{m,n}: = k_0L - \sum_{i=1}^Nk_iE_i$ with the property that
\begin{equation}\labell{eq:V}
V_{m,n}\cdot h(\vareps_N) = 1,\quad V_{m,n}\cdot h(\vareps_i) = 0,\;\; 0\le i<N.
\end{equation}
(This vector $V_{m,n}$ will give the weights of the corresponding ball embedding problem.)
For example, if $(m,n) = (7,12)$ then the conormals $(p_i,q_i)$ to the sequence of edges starting at $\vareps_1=(0,1)$ and going to $\vareps_0=(1,1)$ 
have slopes 
\begin{equation}\labell{eq:mn}
\frac pq = \frac 01,\;\; \frac 12, \;\;\frac 47,\;\;\frac 7{12},\;\;\frac 35,\;\;\frac 23,\;\;\frac 11.
\end{equation}
These edges are numbered 
$$
\ve_1,\ve_2,\ve_5,\ve_6,\ve_4,\ve_3,\ve_0
$$
according to the order of the blow ups.
They have classes 
$$
\begin{array}{lllllll}
 h(\ve_6)=E_6,&& h(\ve_5)=E_5-E_6,&&h(\ve_4)= E_4-E_5-E_6,\\
 h(\ve_3)= E_3-E_4,&& h(\ve_2)=E_2-E_3-E_4-E_5,&&
 h(\ve_1)=E_1-E_2,\\
 && h(\ve_0)= L-E_1-E_2-E_3&&\end{array}
$$
as one can check by performing the relevant blow ups.
Hence,
\begin{equation}\labell{eq:V7}
V_{7,12} = 12 L - 5E_{12}-2E_{34}-E_{56},
\end{equation}
where $E_{j\dots k}: = \sum _{i=j}^k E_i$
Note that $V_{7,12}^2 = 7\cdot 12 = 84$. 
This is a general fact, which is  known in other contexts; cf.
Remark \ref{rmk:cf}.  However we include a proof
here  for completeness.

\begin{lemma} \label{le:ki} For each relatively prime pair $(m,n)$ with $0\le m<n$ there is a unique primitive integral  homology class $V_{m,n}: = nL - \sum_{i=1}^Nk_iE_i$ 
satisfying (\ref{eq:V}).  Moreover, $V_{m,n}^2 = mn$.
\end{lemma} 
\begin{proof} Let us call the coefficient $k_i$ of $E_i$ in $V_{m,n}$ the {\it label} of the edge $\vareps_i$. 
Equation (\ref{eq:hi}) implies that $V_{m,n}$ will satisfy   (\ref{eq:V}) if  the labels 
 $k_N, k_{N-1},\dots, k_0$ are assigned as follows.
  
\begin{quote} {\it Set $k_N=1$.  Given
 $k_j, j>i$, define $k_i$  to be the sum of the labels of the edges $\vareps_j, j\in S_i$.}
\end{quote}

Since the classes $h(\vareps_i), i=0,\dots, N,$ generate $H_2(X_N)$
there is obviously a unique $V_{m,n}$ satisfying (\ref{eq:V}).  Therefore, it remains to check that $k_0=n$ and that $V_{m,n}^2=mn$. 

Using induction, we shall show that for all $(m,n)$ such that $m+n\le K$ then $V_{m,n}$ as defined above
 has the required properties.  Moreover if $(m',n')$ is another pair with $m'+n'\le K$  and if $|mn'-m'n|=1$ then
\begin{equation}\labell{eq:VV}
2V_{m,n}\cdot V_{m',n'} = 1 + mn' + m'n.
\end{equation}
The base case is $K=2$.  Then $V_{0,1} = L-E_1$ and $V_{1,1} = L$.
The required properties are easily verified.
  
  Suppose the result is known for $m+n<K$ and consider a pair $(m'',n'')$ with $m''+n''=K$.  Let $(m,n), (m',n')$ be 
 the neighbors of $(m'',n'')$ in the Farey sequence $\Ff_{n''}$, named so that $n<n'$. 
  Then $m+n<K, m'+n'<K$ and also $|mn'-m'n|=1$. (Note that this implies
  $
  |m''n-mn''|=1, |m''n'-m'n''|=1$.)
   To complete the inductive step, we will show that
\begin{equation}\labell{eq:V''}
  V'': = V_{m,n}+ V_{m',n'} - E_{N''}
\end{equation}
  satisfies all the conditions required of $V_{m'',n''}$, where the edge with conormal 
  $(m'', n'')$ is called $\vareps_{N''}$. 
  
  To make sense of this formula, note that because
  $n<n'$, $\frac mn$ occurs as part of the Farey sequence $\Ff_{n'}$.  
  Hence, because $|mn'-m'n|=1$, $\frac mn$ and $\frac {m'}{n'}$ are adjacent in this Farey sequence.
    It follows that all the edges in $\Ee: =\Ee(m,n)$ occur in $\Ee':= \Ee(m',n')$: indeed $\Ee'$ may be obtained from $\Ee$ by repeatedly blowing up at the vertex of the edge $\vareps_N$
  \lq\lq closest" to $\frac {m'}{n'}$.  (For example, if  
  $\frac mn<  \frac{m'}{n'}$ then one blows up at the vertex of $\vareps_N$ closest to the $x$-axis.)  In particular the edge $\vareps_N$ of 
  $\Ee$, when considered as part of $\Ee'$, is adjacent to $\vareps_{N'}$.
  Therefore we may consider the classes $E_i$ that occur in $V: = V_{m,n}$ to be a subset of those occurring in 
  $V': = V_{m',n'}$ so that the above formula for $V''$ makes sense.  Moreover, if $h'(\vareps_i)$ denotes the class in $X_{n'}$ represented by $\vareps_i$ then, for all
 $i\le N$, $h(\vareps_i)$ is the class obtained from   $h'(\vareps_i)$ by setting $E_j=0, j>N$.
 Similarly, $V$ is the class obtained from $V'$ by setting all $E_j, j>N,$ to $0$. 
   
Now observe that, because $(m,n)$ and $(m',n')$ are the neighbors of
$(m'',n'')$ in $\Ff_{n''}$, 
 $\Ee'': = \Ee(m'',n'')$ is obtained from $\Ee'$ by 
 adding one extra edge 
 $\vareps_{N''}$ between $\vareps_{N}$ and $\vareps_{N'}$. Hence, $N'' = N'+1$.
 Further, if we denote the class of $\vareps_i$ in $\Ee''$  by
 $h''(\vareps_i)$,
$$
h''(\vareps_{N'}) = h'(\vareps_{N'}) - E_{N''},\quad 
h''(\vareps_{N}) = h'(\vareps_{N}) - E_{N''},\quad 
h''(\vareps_i) = h'(\vareps_i), \; i\ne N,N'.
 $$
It follows easily that if $V''$ is defined by equation
(\ref{eq:V''}) then the relations (\ref{eq:V}) hold. Hence $k_0''=k_0+k_0' = n+n'=n''$.
Moreover by equation (\ref{eq:VV})
\begin{eqnarray*}
(V'')^2&=& (V+V' - E_{N''})^2 \\
&=& mn + m'n' + 2VV' - 1\\
&=& (m+m')(n+n') = m''\,n''.
\end{eqnarray*}
It remains to check that $VV''$ and $V'V''$ satisfy the analog of
 equation (\ref{eq:VV}).  This is left to the reader. 
  \end{proof}
  
  \begin{example}\rm
   To illustrate this, consider the case 
   $(m'',n'')=(10,17)$ with Farey neighbors
   $(m,n) = (3,5), (m',n') =(7,12)$.  The edges in $\Ee(3,5)$ have slopes
   $$
   \frac pq =  \frac 01,\;\; \frac 12, \;\;\frac 35,\;\;\frac 23,\;\;\frac 11,
$$
and their classes are
$$
\begin{array}{lllll}
E_1-E_2,& E_2-E_3-E_4, &
 E_4, &  E_3-E_4,&  L-E_1-E_2-E_3.
 \end{array}
$$
Thus $V_{3,5} = 5L -2E_{12}- E_{34}$. Therefore,
by  (\ref{eq:V7}),  formula (\ref{eq:V''}) gives
\begin{equation}\labell{eq:1017}
V_{10,17} =   17 L- 7E_{12}-3E_{34}-E_{567}.
\end{equation}
\end{example}

\begin{figure}[htbp] 
   \centering
   \includegraphics[width=4in]{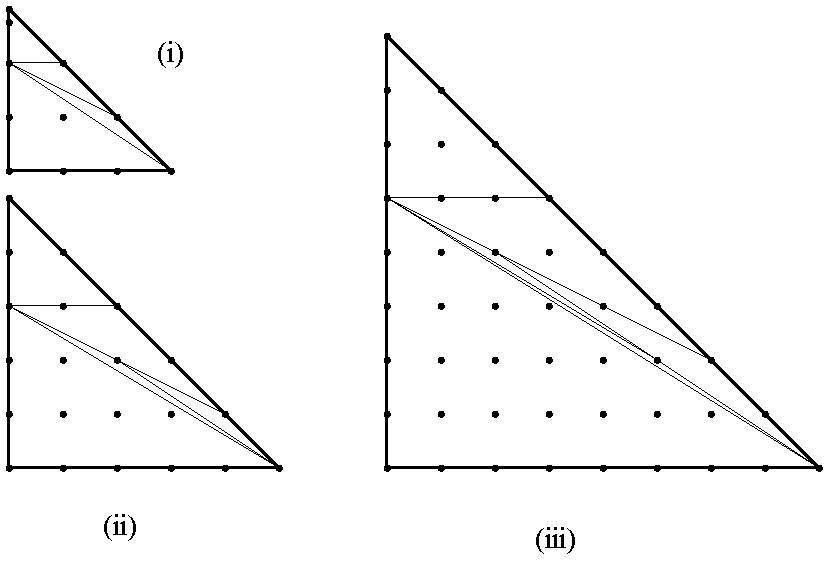} 
   \caption{Blowing up  $\De(n,n)$ to $\De(m,n)$ for $(m,n)$ equal to
   (i)  $(2,3)$, (ii) $(3,5)$ and (iii) $(5,8)$. The points of the integral lattice are marked. Note  that
   $V_{2,3} = 3L-E_{123}$,
   $V_{3,5} = 5L-2E_{12}-E_{34}$ and
   $V_{5,8} = 8L-3E_{12} -2E_3-E_{45}$.  The sizes of the triangles in each decomposition are given by the coefficients of the $E_i$ in each $V_{m,n}$.}
   \label{fig:8}
\end{figure}

In the above discussion each edge $\vareps_i$ was assumed to have positive length, although it was very short for $i<N$. Now imagine performing these cuts so that these edges have zero length.  In other words, at each stage construct $R_i^0$ by cutting out a vertex of $R_{i-1}^0$ together
with the whole of the shorter adjacent edge.  Thus 
each cut 
 $\ga_i^0$ has a vertex at $(0,m)$ or at $(n,0)$ and the end result is $R_N^0: = \De(m,n)$. Thus these cuts
 decompose the triangle $T(m,n): = \De(n,n)\less \ooDe(m,n)$ into a union of triangles, each equivalent to a multiple $\mu_i^0\, \De(1,1)$ of the standard triangle: see  Figure \ref{fig:8}.  We will think of this decomposition of   $T(m,n)$ as corresponding to a singular (nonsmooth) blow up of $\De(n,n)$.

We now show that the multiplicities $\mu_i^0$ are 
precisely the weights $k_i$.  Since the area of the cut triangles 
is $\sum_{i=1}^N (\mu_i^0)^2$ this shows that
$$
n^2-mn = \sum_{i=1}^N k_i^2,
$$
which gives a geometric explanation for the quadratic relation
$V_{m,n}^2=mn$.
 
\begin{rmk}\label{rmk:area}\rm
It turns out that this geometric blow up procedure is very closely related to the construction of an appropriate continued fraction.   This is easiest to see in the context of outer approximations: see Remark~\ref{rmk:cf}.
\end{rmk}

\begin{lemma}\labell{le:V}  Write $ V_{m,n}= nL -\sum k_iE_i$.  
Then $\De(m,n)$ is the (nonsmooth) blow up of $\De(n,n)$ where the cuts have conormals $\nu(\ve_1),\dots,\nu(\ve_N)$ and weights  $k_1,k_2,\dots,k_N$. 
\end{lemma}

\begin{proof}  The coefficient $\mu_i^0$ is just the affine length of the cut $\ga_i$, i.e. it is the length of the corresponding edge of $R_i^0$. This edge represents the class $\mu_i^0 E_i$ in $H_2(X_N)$.  Thus $\mu_N^0 = 1$, the affine length of   
the slanted edge $\ga_N$ of $\De(m,n)$. One can now argue 
that $\mu_i^0 = k_i$ for $i=n_1,n_2,\dots,1$ in turn.  The point is that the $i$th cut leaves an \lq\lq edge" $\vareps_i$ in class $h(\vareps_i)$ of length $0$.  But $\vareps_i$ is the result of cutting $\ga_i$ by cuts of length $\mu_j^0, j\in S_i$.  Therefore the result follows from formula (\ref{eq:hi}) and the definition of the $k_i$ given in the proof of Lemma~\ref{le:ki}.
\end{proof}

\begin{cor} \labell{cor:V}
The open subset of $\C P^2(n)$ with moment polytope $\ooDe(n,n)\less
\De(m,n)$ contains $N$ open balls of sizes $k_1,\dots,k_N$.
\end{cor}
\begin{proof} This holds as in Lemma~\ref{le:open}.
\end{proof}

With this notation in hand, we can now discuss 
 inner approximations with more precision. 
%
As explained abve, an  inner approximation 
to $\De(m,n)$ is obtained by 
moving the edge of $\De(n,n)$ with conormal $(1,1)$ a little closer to the origin (so that it
meets the $x$-axis at the point $(n-\de_0,0)$ for some $\de_0>0$), and then
slightly adjusting the size of all the subsequent blow ups 
from $k_i$ to $k_i+\de_i$ so as not to cut out quite all of an edge at each blow up.  (See Figure \ref{fig:9}. Note that some of the $\de_i$ may be negative; cf. the construction of
$\Hat\Cc_4(\la,\de)$ in \S2.)
For suitable $\de_i$ this will 
 create a smooth polygonal arc $\Ee(m,n;\de)$ whose edges $\ve_i, 0\le i\le N,$ have conormals $\nu(\ve_i)$ as described above.
 
 Define 
\begin{equation}\labell{eq:a}
a_{m,n;\de}: = (n-\de_0)\ell - \sum_{i=1}^{N}(k_i+\de_i)e_i\in H^2(X_N),
\end{equation}
where $e_i$ is Poincar\'e dual to $E_i$, i.e. $e_j(E_i) = -\de_{ij}$.
Since $(n-\de_0)\ell$ is the cohomology class of the symplectic form on $\C P^2(n-\de_0)$, $a_{m,n;\de}$ is the class of the symplectic form  on $X_N$ obtained from $\CP^2(n-\de_0)$ by the  blow up procedure explained above with the $i$th blow up of  size $k_i+\de_i$. 
Hence the affine length $\al(\ve_i)$ of the edge $\ve_i$ is
$$
\al(\ve_i)=a_{m,n;\de}(E_i)>0.
$$
Because $a_{m,n;0}$ is  the Poincar\'e dual of $V_{m,n}$, 
 all the edges of $\Ee(m,n;\de)$ are very short except for $\ve_N$ which has affine length nearly $1$.

 \begin{defn}\labell{def:allow}
 We say that $\de: = (\de_0,\dots,\de_N)$ is {\bf admissible} if:\SSS
 
 \NI (i)  $\de_0,\, \de_1>0$;\SSS
 
 \NI (ii)  the edges $\ve_0,\dots,\ve_N$ of $\Ee(m,n;\de)$ have positive lengths $\al(\ve_i)$;
 \SSS
 
 \NI (iii) $ \Ee(m,n;\de)\;\;\subset\;\; \De(m,n)\less\Bigl(r \,\ooDe(m,n)\Bigr)$ where $r: = 1-\de_0-\de_1$.
 \end{defn}
 
 Note that condition (ii) implies that $\Ee(m,n;\de)$ is a chain of edges with the same intersection properties as $\Ee(m,n)$. Hence the slopes decrease as one moves along $\Ee(m,n;\de)$ from $\ve_1$ to $\ve_0$, so that it is a convex polygonal arc.  Therefore because $\Ee(m,n;\de)$ has endpoints $(0,m-\de_0-\de_1)$ and $(n-\de_0,0)$ where $\de_0, \de_0+\de_1>0$,  it lies outside  
 $\frac {m-\de_0-\de_1}m \;\ooDe(m,n)$.  Therefore,
 to prove (iii) one must simply check that it lies inside $\De(m,n)$.

\begin{lemma}\labell{le:Eede} If $\de$ is admissible, so is 
$t \de$ for all $0<t\le 1$.
\end{lemma}
\begin{proof} Condition (i) in the definition obviously holds.  To check (ii), let us
denote the edges of $\Ee(m,n;\de)$ by $\ve_i^\de.$
Then, if $0\le i<N,$
$$
\al(\ve_i^\de) = a_{m,n,\de}(h(\ve_i)) = (a_{m,n,\de}-
a_{m,n,0}) (h(\ve_i))
$$
is a homogeneous linear function of $\de$ and hence is positive for $t \de$ if it is positive for $\de$. Further $\al(\ve_N^\de) = 1+$
 a homogenous linear function of $\de$, and so again condition (ii) is satisfied by $t\de$.
 
 To check (iii), observe that the positions of the edges depend linearly on $\de$, i.e. for each $i$ there are constants
 $c_{i0}$ and homogeneous linear functions $c_{i1}(\de)$ such that
 the edge $\ve_i^\de$ lies in the line 
 $$
 \{\bx\in \R^2: \nu(\ve_i)\cdot \bx = c_{i0} + c_{i1}(\de)\}.
 $$
 Moreover, by Lemma~\ref{le:V}, the line $\nu(\ve_i)\cdot \bx = c_{i0}$ goes through
 one of the points $(0,m)$ or $(n,0)$.  Hence (iii) holds iff
 $c_{i1}(\de)\le 0$ for all $i$.  The result follows. 
 \end{proof}

There is an analogous discussion for outer approximations.
These  approximate $\De(m,n)$ by a 
polytope with a concave chain $\Hat\Ee(m,n;\de)$ of edges $\Hat \vareps_i$ lying just outside 
the slanted edge of $\De(m,n)$.  We did the case $(1,k)$ in \S2: 
 $\Hat\Ee(1,4)$ (which in the notation of \S2 corresponds to the chain of spheres $\Hat\Cc_4$) is illustrated in Figure~\ref{fig:3}.
Note that this chain of edges goes between the edges 
 with conormals $(-1,0)$ and $(0,-1)$ but does not include them, so that the classes
$h(\Hat \vareps_i)$ of the edges in $\Hat\Ee$ are linear combinations of the exceptional divisors $\Hat E_i$, with no mention of $L$.
Again we assume that $\Hat\Ee(m,n)$ is minimal, i.e. the last edge $\Hat \vareps_{\Hat N}$ with conormal $(-m,-n)$ is the only edge whose class $h(\Hat \vareps_i)$ has self intersection $-1$.  We shall denote the analog of $V_{m,n}$ by $\Hat V_{m,n}$.  Thus $\Hat V_{m,n}=\sum_{i=1}^{\Hat N} \Hat k_i \Hat E_i $ is such that
\begin{equation}\labell{eq:W}
\Hat V_{m,n}\cdot h(\Hat \vareps_{\Hat N}) = -1,\;\;
\Hat V_{m,n}\cdot h(\Hat \vareps_j) = 0,\; j<\Hat N, \;\; 
\Hat V_{m,n}^2 = -mn.
\end{equation}
For example $\Hat V_{1,k} = \sum_{j=1}^k\Hat E_j$.   We leave the proof of the following statement to the reader.

\begin{lemma} \labell{le:outer} Let $(m,n)$ be relatively prime positive integers with $m<n$.  Then:\MS

\NI {\rm (i)}   For small admissible   $\Hat\de$, $\De(m,n)$ has an  outer approximation 
$\Hat \Ee(m,n;\Hat \de)$ that lies in $(1+\Hat\de_0) \,\De(m,n) \less \ooDe(m,n)$, where $\Hat\de_0>0$.\SSS  

\NI {\rm (ii)}  If $\Hat \de$ is admissible, so is $t\Hat\de$ for all $0<t\le 1$.\SSS

\NI{\rm (iii)} There is a vector $\Hat V_{m,n}$ 
satisfying (\ref{eq:W}).
\end{lemma}

\begin{rmk}\labell{rmk:cf}\rm
Just as in the case of inner approximations, the conormals occurring in an outer approximation to $\De(m,n)$ give rise to a decomposition of a triangle, which this time is $\De(m,n)$ itself.  Moreover, as we shall prove in \cite{McS}, the sequence of labels
$\Hat k_1,\Hat k_2,\dots, \Hat k_N$ that occur as coefficients in the vector $\Hat V_{m,n}$
can be obtained by the following version of the Euclidean algorithm.  

\begin{quote}{\it
First write down $a_1$ copies of $p_1: = m$ where $a_1m\le n < (a_1+1)m$,
then write down $a_2$ copies of $p_2: = n-a_1p_1$ where $a_2 p_2\le m=p_1 < (a_2+1)p_2$, and so on.  At the $i$th step one writes down
$a_i$ copies of $p_i: = p_{i-2}-a_{i-1}p_{i-1}$ where $a_ip_i\le p_{i-1} < (a_i+1)p_i$.  The process stops as soon as some $p_i=0$.}
\end{quote}

\begin{figure}[htbp] 
   \centering
   \includegraphics[width=3in]{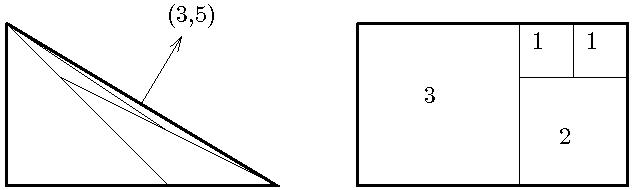} 
   \caption{ $\frac {5}{3}$ has labels $3,2,1,1$ with multiplicities
   $1,1,2$.  Note that in the diagram on the right one starts by expanding horizontally because the rectangle is wider than it is high; at the second step one rotates by $90^o$ and then continues.  This rotation is equivalent to taking the reciprocal of the aspect ratio of the rectangle.  Hence this expansion mirrors the continued fraction.}
   \label{fig:53}
\end{figure}
\NI
As pointed out Dylan Thurston \footnote{Private communication.}  the combinatorics of the resulting decomposition of $\De(m,n)$  are precisely the same as the combinatorics of one of the standard ways of getting
 the continued fraction expansion: see Figure \ref{fig:53}.
 Hence the multiplicities $a_i$ of the labels $\Hat k_j$ give the 
continued fraction expansion of $\frac mn$.   For example,
 $\frac {5}{3}$ has labels $\Hat k_1,\dots,\Hat k_4 = 3,2,1,1$ with multiplicities
   $a_1,a_2,a_3 = 1,1,2$, and 
$$
\frac 53 = 1 + \frac{1}{1+\frac{1}{2}}.
$$
Note also that the multiplicative relation $\Hat V_{m,n}^2 = \sum \Hat k_i^2 =mn$ is obvious from this point of view. Since, as
we show in Theorem \ref{thm:ell2} below, the labels $\Hat k_i$
determine the
weights of the corresponding ball embedding problem, this
multiplicative relation  corresponds to the geometric fact that the total volume of the  balls corresponding to an ellipsoid
$E$ must be the same as the volume of $E$.
\end{rmk}
%

\begin{rmk}\label{rmk:sing} \rm
(i)  We saw above that an inner approximation to $\De(m,n)$ is 
constructed from a decomposition of the
 triangle $T(m,n)=\De(n,n)\less \ooDe(m,n)$ with outward conormals $(1,1), (-1,0), (-m,-n)$, while
an outer approximation to $\De(m,n)$ is constructed from a decomposition of 
$\De(m,n)$ itself.  The matrix
$$
A=\left(\!\!\begin{array}{rr} 0&-1 \\1 &-1 \end{array}\!\!\right)
$$
takes the conormals of $T(m,n)$ to those of $\De(n,n-m)$. It is easy to check that this transformation takes the first decomposition
into the second.  Thus inner and outer approximations are essentially
the same thing, though they are related in a slightly different way to
the ambient triangle $\De(n,n)$.
\SSS

\NI (ii) From the point of view of singularity theory, our construction of
the inner (or outer) approximation to $\De(m,n)$ can be considered as a kind of joint resolution
 of the two singular points of the corresponding toric variety.  Usually, one would resolve them separately, in which case, it is the Hirzebruch-Jung continued fractions (with $-$ rather than $+$  signs) that are relevant: see Fulton \cite[\S2.6]{F}.   In the standard resolution of a single singularity
 one performs the blow ups near just one of the vertices getting  half of our conormals.   For example Fulton's method of resolving the
  vertex with outward conormals $(0,1)$ and $(n,-m)$ (where $0<m<n$)
  begins by a cut with conormal $(1,0)$.  Hence if we rotate his picture 
   anticlockwise by $90^\circ$ we get the
 half of the inner approximation to $\De(m,n)$ near the vertex 
 $(0,m)$; cf. Figure \ref{fig:6}.  Note that the orbifold structure of this vertex has stabilizer of order $n$, with generator $\zeta$ acting on $\C^2$ via $(z_1,z_2)\mapsto (\zeta^{-m} z_1, \zeta z_2)$ 
 where $\zeta=e^{2\pi i/n}$.
 The other half of this inner approximation corresponds to the vertex
 $(n,0)$ which has stabilizer of order $m$ acting via 
 $(z_1,z_2)\mapsto (\eta z_1, \eta^{-n} z_2)$ where $\eta=e^{2\pi i/m}$.
 Hence this corresponds to
Fulton's resolution of the  vertex with outward conormals $(0,1)$ and 
 $(m, km-n)$ where we choose $k$ so that $0<n-km<m$; i.e. we simply interchange the roles of $m$ and $n$.
Note that the interpretation of the coefficients of the continued fraction expansion
 is rather different in the two cases.
\end{rmk}

\subsection{Proof of Theorem~\ref{thm:ell}.}\labell{ss:thm}

We shall prove the following more precise 
form of Theorem \ref{thm:ell}. Recall that $V$ and $\Hat V$ are defined in equations (\ref{eq:V}) and (\ref{eq:W}) respectively.

\begin{thm}\labell{thm:ell2} Suppose that each pair $(m,n)$ and $(m',n')$ is mutually prime, and let 
$$
\Hat V_{m,n} = \sum_{1\le i\le \Hat N} \Hat k_i\Hat E_i,\quad
V_{m',n'} = n'L -\sum_{1\le i\le N}  k_i E_i.
$$
Set
$$
k = N + \Hat N, \quad {\un w}_\la = 
\left(\frac{k_1}{n'},\dots,\frac{k_N}{n'},
\frac{\la \Hat k_1}{n'},\dots,\frac{\la \Hat k_{\Hat N}}{n'}\right).
$$
Then  the question of whether $\la E(m,n)$ embeds into the open ellipsoid $\oE(m',n')$ is equivalent to the symplectic packing problem for $k$ balls with weights ${\un w}_\la $.
\end{thm}
\begin{proof}  Suppose first that $E(\la m,\la n)$ embeds into the open ellipsoid $\oE(m',n')$.  Since $\oE(m',n')\subset \oB(n')$ we may consider $\oE(m',n')$ as a subset of
$\C P^2(n')$.  By Corollary~\ref{cor:V} the complement of
$\oE(m',n')$ in  $\C P^2(n')$ contains the $N$ open balls $\oB(k_i)$ (cf. Lemma \ref{le:open}.)  Moreover these balls
can be embedded disjointly from the line in $\C P^2(n')$ represented by the edge of $\De(n',n')$ with conormal $(1,1)$.
Similarly, $\la E(m,n)$ 
contains the $\Hat N$ open balls $\la \oB(\Hat k_j)$.  Hence, rescaling by $1/n'$, we see that
the given symplectic packing problem has a solution with open balls.

However, the problem was formulated in terms of embedding closed balls.  To deal with this, observe that when $\la E(m,n)\se \oE(m',n')$ there is  $\ka>0$ such that
 $(1+\ka)\la E(m,n)$ embeds in $\frac 1{1+\ka}\oE(m',n')$.
 This means that the sizes of  all the open balls
can be slightly increased so that they contain
 closed balls of the correct size.

Conversely, suppose that the ball packing problem has a solution, i.e. that there is a symplectic form on the $k$-fold blow up $X_k$  of $\C P^2$ in class 
\begin{equation}\labell{eq:ala}
a_\la: = \ell - \sum_{i=1}^k w_ie_i. 
\end{equation}
Since the 
space of symplectic forms is open, we may suppose 
without loss of generality that $\la$ is rational.  Moreover, it suffices to prove that 
\begin{equation}\labell{eq:la0}
\la_0\,E(m,n)\;\se\; \oE(m',n')\;\;\mbox{ for all }\la_0<\la.
\end{equation}

Before proceeding further, it is convenient to introduce some notation.
We will denote the divisors of $X_k$ by $E_1,\dots, E_k$ as usual; hence the $E_{N+i}, 1\le i\le \Hat N$, correspond to the exceptional divisors $\Hat E_i$ associated to $E(m,n)$.
Further $\de$ will denote a tuple of small constants, whose length 
(either $N+1,\Hat N$ or $k+1=N+1+\Hat N$) will depend on the context.
When it is necessary to be more specific we shall denote the $(k+1)$-tuple $\de$ by $(\de',\Hat\de)$. Further
 $\de = (\de',\Hat\de)$ is admissible if its first $N+1$ components $\de'$ are admissible  for $E(m',n')$ while its last $\Hat N$ components $\Hat\de$ are admissible  for $E(m,n)$.

Given an inner approximation $\Ee(m',n';\de)$, we shall
denote by $U_\de$ the $T^2$ invariant open subset of $\CP^2(n')$ whose moment image is  the component of $\De(n',n')\less \Ee(m',n';\de)$ that lies in $\De(m',n')$.
Thus $U_\de\subset \oE(m',n')$ is a smooth approximation to $E(m',n')$.
We shall denote the chain of spheres corresponding to 
an inner approximation $\Ee(m',n';\de)$ by $\Cc_\de$ and that corresponding to an outer approximation $\Hat \Ee(m,n;\de)$ by $\Hat\Cc_\de$.  

Finally if $U$ is any subset of a symplectic manifold  $(M,\om)$ and $r>0$ we shall denote by $r U$ the set $U$ provided with the form $r\om|_U$. Note that if $\Om$ is a symplectic form on $X_k$ such that the disjoint union
$\Cc_\de\sqcup \la \Hat\Cc_\de$ embeds in $(X_k,\Om)$, then the class $[\Om]$ is determined by $\de$ and is close to $n(\ell-\sum w_ie_i)$.
 \MS

\NI {\bf Claim 1:}  {\it If there is a symplectic form $\Om$ on $X_k$ such that
$
\Cc_\de\sqcup \la_0 {\Hat \Cc}_\de$ embeds in $(X_k,\Om)$ for some admissible $\de$, then
 $\la_0E(m,n)\se \oE(m',n')$.}\MS

\NI {\it Proof.}  Let $\Nn(\Cc_\de)$ be a $T^2$-invariant 
neighborhood of 
$\Cc_\de$ whose moment image is a neighborhood of 
$\Ee(m',n';\de)$.  As in Lemma \ref{le:nb} we may suppose that 
$\Nn(\Cc_\de)$ embeds in $(X_k,\Om)$.  Then a neighborhood 
of infinity in $(W, \Om): = (X_k\less \Cc_\de,\Om)$ may be identified with 
  a neighborhood 
of infinity in $U_\de$, where $U_\de$ is as above. 
(In fact, $(X_k\less \Cc_\de,\Om)$ can be obtained from $U_\de$ by further blowing up near the inverse image of $(0,0)$.)   Moreover $(W,\Om)$ contains a copy of $\la_0 \Hat\Cc_\de$ and hence, as in Lemma~\ref{le:1}, blows down to an open set $(Z,\om)$ containing $\la_0E(m,n)$. But $H_2(Z) = 0$ by construction, and $(Z,\om)$ is symplectomorphic to
$U_\de$ at infinity.  Hence, by the uniqueness of symplectic forms on starshaped subsets of $\R^4$ that are standard near the boundary
(see \cite[Thm.~9.4.2]{MS}), $(Z,\om)$ is symplectomorphic to $U_\de$.
Therefore
$$
\la_0E(m,n)\;\se\; (Z,\om)\;\cong\; U_\de\;\subset\; E(m',n'),
$$ 
which proves the claim.\QED

\NI {\bf Claim 2:}  {\it If the ball packing problem has a solution with weights ${\un w}_\la$, then for all $\la_0<\la$, there is a symplectic form $\Om$ on $X_k$ such that
$
\Cc_\de\sqcup \la_0 {\Hat \Cc}_\de$ embeds in $(X_k,\Om)$ for some admissible $\de$.}\MS

\NI {\it Proof.}\,\,
If $r< \frac{m'}n$ then $r E(m,n)$ embeds linearly in
$\oE(m',n')$ and so, for small admissible $\de$ there is 
a symplectic form $\Om_{r,\de}$ on $X_k$ such that
$
\Cc_\de\sqcup r {\Hat \Cc}_\de$ embeds in $(X_k,\Om_{r,\de})$.
Note that $[\Om_{r,\de}] = a' + r \,\Hat a + c(\de)$, where 
$$
a': = n'\ell - \sum_{1\le i\le N}w_ie_i,\quad \Hat a: =  -
\frac1 {\la}\,\sum_{N< i\le k}{w_i}e_i,
$$
and  $c(\de) = c'(\de') + \Hat c(\Hat\de)\in H^2(X_k)$ is a homogeneous linear function of $\de$.

By assumption the class $a_\la:=n'\ell- \sum w_ie_i = a'+\la\Hat a$ 
of equation (\ref{eq:ala}) is rational and represented by a symplectic form.
 Therefore, as in the proof of Proposition~\ref{prop:M}, the homology class $PD(qa_\la)$ has nontrivial Gromov invariant for large $q$. Choose an $\Om_{r,\de}$ tame almost complex structure $J$ on $X_k$ such that  both $\Cc_\de$ and $ r{\Hat \Cc}_\de$ are $J$-holomorphic.  If $J$ is sufficiently generic, then as in the proof of Theorem~\ref{thm:1} in \S2 we may suppose that 
the class $PD(qa_\la)$ is represented by a connected $J$-holomorphic submanifold $Q$ that intersects $\Cc_\de$ and $ r {\Hat \Cc}_\de$
transversally.  The inflation procedure gives a family of symplectic forms $\Om_t$ on $X_k$ that are nondegenerate on the two  configurations of spheres and lie in class 
\begin{eqnarray*}
[\Om_t] &=&[\Om_{r,\de}] + tqa_\la\\
&=&  (1+tq)\left(a'+ \bigl(\frac{ r + \la tq}{1+tq}\bigr)\,\Hat a+c'(\ka')+\Hat c(\Hat\ka)\right) 
\end{eqnarray*}
where $\ka'$ and $\Hat \ka$ are  multiples of $\de'$ and $ \Hat \de$. 
and so  are admissible by Lemmas~\ref{le:Eede} and \ref{le:outer}.  
Observe that as $t\to \infty$ the class $\frac 1{1+qt}[\Om_t]$ 
converges to  $a_\la$.  Moreover, for appropriate $\de$, $\Cc_\de$ embeds in $(X_k,\Om_t)$ for all $t$, while $\la_0 \Hat\Cc_\de$ embeds in $(X_k,\Om_t)$ if $\la_0= 
\frac{ r + \la tq}{1+tq}$.  By equation (\ref{eq:la0}) this completes the proof.
\end{proof}

\NI {\bf Proof of Corollary \ref{cor:conn}.} Since the targets of the embeddings are open ellipsoids, an easy continuity argument implies that it suffices to prove these statements when $a,b,a',b'$ are integers.
 Part  (i) is equivalent to saying that 
all deformation equivalent symplectic forms on $X_k\less \Nn(\Cc_\de\cup\Hat \Cc_\de)$ are isotopic.  It can be proved in the same way as the uniqueness of symplectic forms on $X_k$. One just needs to inflate along curves $Q$ that intersect $\Cc_\de$ and $\Hat\Cc_\de$ transversally, which is possible as in the proof of Claim 2 above.  For more details, see \cite{M}.

(ii) follows from (i) just as the analogous statement for balls follows from the fact that the space of embeddings of one ball into another is connected.  Let $\la_n, n\ge1,$ be an increasing sequence with limit $\la$.  From a sequence of embeddings
$$
\io_n: \la_n\,E(m,n)\;\se\; \oE(m',n')
$$
one first uses (i) with target ellipsoid $ \oE(m',n')$ to construct a sequence $\io_n'$ such that ${\rm Im}\,\io_n'\subset {\rm Im}\,\io_{n+1}'$.  By using (i) again, this time with
target ellipsoids ${\rm Im\,}\io_{n+1}'$, one makes
 a further adjustment so that $\io_{n+1}'$ restricts to $\io_n'$ on 
$\la_n\,E(m,n)$.  The result follows.
\QED

\NI {\bf Proof of Proposition \ref{prop:tol} (i).}
We saw above that $V_{2,3} = 3L - F_1-F_2-F_3$ and $\Hat V_{1,4} = \Hat E_1+\Hat E_2+\Hat E_3+\Hat E_4$.  Hence the first statement follows from Theorem~\ref{thm:ell2}.

For simplicity, let us rename the blow up classes in $X_7$ as $E_i, i=1,\dots, 7,$ where $F_i: = E_{i}$ for $1\le i\le 3$ and $E_i: = \Hat E_{i-3}$ for $4\le i\le 7$. By Propositions~\ref{prop:MP} 
and \ref{prop:M}, the second statement will follow if we show that
the class 
$$
a_{\un w} = \ell-\frac 13 e_{123} - \frac \la 3 e_{4567}
$$
takes positive values on all the elements in $\Ee_K(X_7)$, where 
$e_{j\dots k}: = \sum_{i=j}^ke_i$.  But $\Ee_K(X_7)$ 
is generated by classes of the form
 $E_i$, $L-E_i-E_j$ together with classes that are equivalent to the following (after permutation of indices):
 $$
{\rm (i)}\quad 2L-E_{1\dots 5},\qquad {\rm (ii)}\quad
 3L-2E_1-E_{2\dots 7}.
 $$
Evaluating $a_{\un w}$ on $ 3L-2E_7-E_{1\dots 6}$ we find that
we need 
$3>\frac 53\la +1$ i.e. $\la <\frac 65.$  Since the other curves in $\Ee_K(X_7)
$ give weaker inequalities, the result follows.\QED

\NI {\bf Proof of Proposition \ref{prop:tol} (ii).}
The first statement holds as before.  To prove the second, recall
 that $\Ee_K(X_8)$ is generated by the classes  in $\Ee_K(X_7)$ 
together with those of the following three forms:
$$
\begin{array}{ll}
{\rm (iii)}& 4L - 2E_{123} - E_{4\dots 8};\\
{\rm (iv)} & 5L - 2E_{1\dots 6} - E_{78};\\
{\rm (v)} & 6L - 3E_1 - 2E_{2\dots8}.
\end{array}
$$
(These structural results on $\Ee_K(X_7)$ and 
$\Ee_K(X_8)$ are classical and not hard to prove directly from the definition.)
One gets the sharpest inequality on $\la$ from 
elements of the form (v), which give 
$\la< \frac {12}{11}$.  Hence the result.\QED


\begin{thebibliography}{cccccc}

\bibitem{B} P. Biran,
      Symplectic packing in dimension $4$,
      {\it Geometric  and Functional Analysis\/}, {\bf 7} 
      (1997), 420--37.

\bibitem{BE} P. Biran, From symplectic packing to algebraic geometry and back,   {\it European Congress of Mathematics, Vol II, (Barcelona 2000),}
507--524, {\it Progr. Math.} {\bf  202}, Birkh\"auser, Basel 2001.

\bibitem{CHLS} 
K. Cieliebak, H. Hofer, J. Latschev and F. Schlenk,
Quantitative symplectic geometry, arXiv:math/0506191,  {\it Dynamics, Ergodic Theory, Geometry MSRI}, {\bf 54} (2007), 1--44.

\bibitem{F}  W. Fulton, {\it Introduction to Toric Varieties},
 Annals of Math Studies vol 131, PUP (1993).
 
\bibitem{God}  L. Godinho, Blowing up symplectic Orbifolds,
{\it Ann. Global Anal. Geom.}  {\bf 20} (2001), 117-62.

\bibitem{G} 
M. Gromov,  Pseudo holomorphic curves in symplectic manifolds, 
        {\it Inventiones Mathematicae\/}, {\bf 82} (1985), 307--47. 

\bibitem{Gu}  L. Guth, Symplectic embeddings of polydiscs, 
      arXiv:math/0709.1957, {\it Invent. Math.\/}, {\bf 172} (2008), 477--489.

\bibitem{HW}  G.H. Hardy and E.M. Wright, 
      {\it An Introduction to the Theory of Numbers}, 
      OUP, Oxford (1938).

\bibitem{HLR} J. Hu, T.-J. Li and Yongbin Ruan, Birational 
        cobordism invariance of uniruled symplectic manifolds, 
        arXiv:math/0611592 to appear in {\it Invent. Math.}.

\bibitem{K}   Yael Karshon; Appendix to \cite{MP},
      Invent. Math.  115 (1994), 431--434. 

\bibitem{KM} P. Kronheimer and T. Mrowka, The genus of embedded
 surfaces in the projective plane, {\it Math. Res. Letters}, (1994), 797--808

\bibitem{LL2} Bang-He Li and T.-J. Li, Symplectic genus, minimal genus and diffeomorphisms, {\it Asian J. Math.} {\bf 6} (2002), 123-44.
  
\bibitem{LL} Tian-Jun Li and A. K. Liu, Uniqueness of symplectic
     canonical class, surface cone and symplectic cone of $4$-   
     manifolds with $b^+=1$, {\it J. Diff. Geom.} {\bf 58} (2001),    
       331--70.

\bibitem{Mcuniq}  D. McDuff, Blowing up and symplectic embeddings in
dimension $4$. {\it Topology\/}, {\bf 30}, (1991), 409--21. 


\bibitem{M} D. McDuff,  From  symplectic deformation to  isotopy, 
     {\it Topics in Symplectic $4$-manifolds (Irvine CA 1996),} ed.    
     Stern, Internat. Press, Cambridge, MA (1998), pp 85-99.

\bibitem{Me} D. McDuff, Some $6$ dimensional Hamiltonian $S^1$-manifolds,  arxiv:0808.3549.


\bibitem{MP} 
D. McDuff and L. Polterovich,  
Symplectic packings and algebraic geometry,
{\it Inventiones Mathematicae}, {\bf 115} (1994), 405--29. 

\bibitem{MS}  D. McDuff and D.A. Salamon,  
        {\it $J$-holomorphic curves and symplectic topology}.    
        Colloquium Publications {\bf 52}, American 
        Mathematical Society, Providence, RI, (2004).  

\bibitem{McS} 
D. McDuff and F. Schlenk, in preparation.

\bibitem{Op}  E. Opshtein, Maximal symplectic packings of $\PP^2$,
arxiv:0610677, {\it Compos. Math.} {\bf 143} (2007), 1558--1575.

\bibitem{Sch0} F. Schlenk, {\it Packing Symplectic manifolds by hand}, 
{\it J. Ssymplectic Geom.} {\bf 3} (2005), 313-40.

\bibitem{Sch} F. Schlenk, {\it Embedding problems in symplectic 
      geometry}, De Gruyter Expositions in Mathematics, de Gruyter 
      Verlag, Berlin (2005) 
      see also ftp://ftp.math/ethz.ch/users/schlenk/buch.ps

\bibitem{Sym} M. Symington, Symplectic rational blowdowns,
{\it J. Diff. Geom.} {\bf 50} (1998), 505--18.

\bibitem{Tau}  C. H.~Taubes, The Seiberg--Witten and the Gromov 
invariants,
     {\it Math. Research Letters} {\bf 2}, (1995), 221--238.

\bibitem{Tol}  S. Tolman, On a symplectic generalization of Petrie's conjecture, preprint (2007).

\bibitem{T} L. Traynor, Symplectic packing constructions, 
      {\it J. Diff. Geom.} {\bf 42} (1995), 411-29.

\bibitem{W} I. Wieck, Explicit symplectic packings, Ph. D. thesis, 
Universit\"at zu K\"oln (2008).
\end{thebibliography}
\end{document}